\documentclass[12pt]{amsart}
\usepackage{amsfonts,latexsym,rawfonts,amsmath,amssymb,amsthm,a4wide}
\usepackage{graphicx}
\usepackage[plainpages=false]{hyperref}
\numberwithin{equation}{section}

\newcommand{\pd}[2]{\frac {\partial #1}{\partial #2}}
\newcommand{\al}{\alpha}
\newcommand{\bb}{\beta}
\newcommand{\la}{\lambda}
\newcommand{\La}{\Lambda}
\newcommand{\oo}{\omega}
\newcommand{\Om}{\Omega}

\newcommand{\dd}{\delta}
\newcommand{\Na}{\nabla}

\newcommand{\ee}{\epsilon}

\newcommand{\beq}{\begin{equation}}
\newcommand{\eeq}{\end{equation}}
\newcommand{\beqs}{\begin{eqnarray*}}
\newcommand{\eeqs}{\end{eqnarray*}}
\newcommand{\beqn}{\begin{eqnarray}}
\newcommand{\eeqn}{\end{eqnarray}}
\newcommand{\beqa}{\begin{array}}
\newcommand{\eeqa}{\end{array}}

\def\td{\tilde}
\def\p{\partial}

\def\ZZ{{\mathbb Z}}
\def\QQ{{\mathbb Q}}
\def\RR{{\mathbb R}}

\def\PP{{\mathbb P}}
\def\CC{{\mathbb C}}

\def\pbp{\frac {\sqrt{-1}}2 \partial\bar\partial}
\def\pbps{  \sqrt{-1}\partial\bar\partial}

\def\osc{{\rm osc\,}}

\def\cH{{\mathcal H}}

\def\cS{{\mathcal S}}

\def\cU{{\mathcal U}}

\def\i{{\sqrt{-1}}}

\def\Aut{{\rm Aut}}
\newtheorem{prop}{Proposition}[section]
\newtheorem{theo}[prop]{Theorem}
\newtheorem{lem}[prop]{Lemma}

\newtheorem{cor}[prop]{Corollary}
\newtheorem{rem}[prop]{Remark}

\title{A criterion for the properness of the $K$-energy in a general K\"ahler class}

\author{Haozhao Li$^1$  }
\address{ Department of Mathematics, University of Science and Technology
of China, Hefei, 230026, Anhui province, China and Wu Wen-Tsun Key
Laboratory of Mathematics, USTC, Chinese Academy of Sciences, Hefei
230026, Anhui,  China} \email{hzli@ustc.edu.cn}

\author{Yalong Shi$^2$}
\address{
Department of Mathematics and Institute of Mathematical Science, Nanjing University, Nanjing, 210093,
Jiangsu province, China}
\email{shiyl@nju.edu.cn}

\author{Yi Yao}
\address{
Department of Mathematics and Institute of Mathematical Science, Nanjing University, Nanjing, 210093,
Jiangsu province, China}
\email{yeeyoe@163.com}

\thanks{$^1$Research
partially supported by NSFC grant No. 11001080 and No. 11131007.}

\thanks{$^2$Research partially supported by NSFC grants No. 11101206, No. 11171143 and by a Project Funded by the Priority Academic Program Development of Jiangsu Higher Education Institutions.}

\begin{document}

\bibliographystyle{plain}

\date{}

\maketitle

\begin{abstract} In this paper, we give a criterion for the properness
 of the $K$-energy in a general K\"ahler class of a compact K\"ahler
 manifold by using Song-Weinkove's result in \cite{[SW]}.  As applications, we give some K\"ahler classes on
$\mathbb{C}\mathbb{P}^2\#3\overline {\mathbb{C}\mathbb{P}^2}$
and $\mathbb{C}\mathbb{P}^2\#8\overline {\mathbb{C}\mathbb{P}^2}$ in which the $K$-energy is proper.
Finally, we prove Song-Weinkove's result on the existence of critical points of $\hat J$ functional by the
continuity method.

\end{abstract}

\tableofcontents
\section{Introduction}

 The
behavior of the $K$-energy plays an important role in
K\"ahler geometry. It is conjectured by Tian \cite{[Tian]} that there
exists a constant scalar curvature K\"ahler (cscK) metric  in a K\"ahler class
$\Om$ if and only if the $K$-energy is proper on $\Om$. For the
K\"ahler-Einstein case, this was proved by Tian when $M$ has no
nontrivial holomorphic vector fields. For the general case, Chen-Tian
\cite{[CT]} showed that the $K$-energy is bounded from below if $M$
has a cscK metric. On toric manifolds, using Donaldson's idea in
\cite{[Don2]} Zhou-Zhu \cite{[ZhZh]} gave a sufficient condition on
the properness of the (modified) $K$-energy on the space of invariant potentials. In a series of papers
\cite{[Paul1]}\cite{[Paul2]}\cite{[Paul3]}, S. Paul gave a sufficient and necessary condition on
the  lower boundedness and properness of
the $K$-energy on the finite dimensional  spaces of Bergman metrics. However,  it is still difficult to
analyze the behavior of
the $K$-energy on  general K\"ahler
manifolds. In this paper, we give a sufficient condition for the properness
 of the $K$-energy in a general K\"ahler class on a compact K\"ahler
 manifold by using the $J$-flow. \\

The $J$-flow was introduced by Donaldson \cite{[Don1]} and Chen
\cite{[Chen1]} independently, and it was used to obtain the properness or the
lower bound of the $K$-energy on a compact K\"ahler manifold with
negative first Chern class.
As pointed out by Chen \cite{[Chen1]}, the $J$-flow is a gradient
flow of the functional $\hat J_{\oo,
\chi_0}$, which  is strictly convex  along any $C^{1, 1}$ geodesics. Thus,
if there is a critical point of $\hat J_{\oo,
\chi_0}$ in a K\"ahler class, then $\hat J_{\oo,
\chi_0}$ is bounded from below and   the $K$-energy is proper when the first Chern class is negative by the formula of $K$-energy relating $\hat J$.
Therefore, to obtain the properness of the $K$-energy it suffices to
study the existence of the critical point of $\hat J_{\oo,
\chi_0}$. In \cite{[SW]} Song-Weinkove gave a sufficient and necessary condition
to this problem, and their result directly implies that the
$K$-energy is proper on a K\"ahler class $[\chi_0]$ on a
$n$-dimensional K\"ahler manifold $X$ of $c_1(X)<0$ with the
property that there is a K\"ahler metric $\chi'\in [\chi_0]$ such
that
\beq
\Big(-n\frac { c_1(X)\cdot
[\chi_0]^{n-1}}{[\chi_0]^n}\chi'+(n-1)c_1(X)\Big) \wedge \chi'^{n-2}>0.
\label{eq:109}
\eeq
Moreover,
Song-Weinkove asked whether the conclusion still holds if the
inequality (\ref{eq:109}) is not strict. In  \cite{[FLSW]}
 Fang-Lai-Song-Weinkove  studied the $J$-flow on the boundary of the K\"ahler cone and
 gave an affirmative answer in complex dimension 2. Later,
 Song-Weinkove \cite{[SW2]} gave a result on the properness of the
 $K$-energy on a minimal surface with general type. Besides,
 Lejmi-Sz\'ekelyhidi \cite{[LS]} studied the relation of the
 convergence of $J$-flow to a notion of stability.\\

To state our main results, we recall Tian's $\alpha$-invariant for a K\"ahler class
$[\chi_0]$:
$$\al_X([\chi_0])=\sup\Big\{\al>0\;\Big|\;\exists \,C>0,\;
 \int_X\; e^{-\al(\varphi-\sup\varphi)}\chi_0^n\leq C,
 \quad \forall\;\varphi\in \cH(X,  \chi_0 )\Big\},$$
 where $\cH(X,  \chi_0 )$ denotes the space of K\"ahler potentials
 with respect to the metric $\chi_0$.
For any compact subgroup $G$ of $\Aut(X)$, and a $G$-invariant K\"ahler class $[\chi_0]$, we can similarly define the $\alpha_{X,G}$ invariant by using $G$-invariant potentials in the definition.

\begin{theo}\label{theo:main1}Let $X$ be a $n$-dimensional compact K\"ahler manifold.
If the K\"ahler class $[\chi_0]$ satisfies the following conditions
for some constant $\ee:$
\begin{enumerate}
  \item[(1)] $0\leq \ee<\frac {n+1}{n}\al_X([\chi_0]), $
  \item[(2)] $\pi c_1(X)<\ee [\chi_0],$
  \item[(3)]   \beq
 \Big(-n\frac {\pi c_1(X)\cdot
[\chi_0]^{n-1}}{[\chi_0]^n}+\ee\Big)[\chi_0]+(n-1)\pi c_1(X)  >0,
\nonumber
\eeq
\end{enumerate}
then the $K$-energy is proper on $\cH(X, \chi_0).$ If instead of (1), we assume  $[\chi_0]$ is $G$-invariant for a compact subgroup $G$ of $Aut(X)$ ,  and $0\leq \ee<\frac {n+1}{n}\al_{X,G}([\chi_0])$, then  the $K$-energy is proper on the space of $G$-invariant potentials.
\end{theo}

In the polarized algebraic case,  our theorem translates into the following
\begin{cor}\label{cor:main1} Let $X$ be a $n$-dimensional compact K\"ahler manifold
and $L$ an ample holomorphic line bundle on $X$. If there is a positive number $\ee>0$ such that the following
conditions hold:
\begin{enumerate}
  \item $\al_X(\ee\pi c_1(L))>\frac{n}{n+1},$ (or equivalently, $\al_X(\pi c_1(L))>\frac{n\ee}{n+1},$)
  \item $ K_X+\ee L>0,$
  \item  \beq
\frac {nK_X\cdot
L^{n-1}+\ee L^n}{L^n}L-(n-1)K_X>0,\nonumber
\eeq
\end{enumerate}
then the $K$-energy is proper on the K\"ahler class $\pi c_1(L)$.
\end{cor}

A direct corollary of Theorem \ref{theo:main1} is the following
result, which gives a partial answer to the question of Song-Weinkove in \cite{[SW]}
and generalize a result of Fang-Lai-Song-Weinkove \cite{[FLSW]} to
higher dimensions.

\begin{cor}\label{cor:main2}Let $X$ be a compact K\"ahler manifold with
$c_1(X)<0$.
If the K\"ahler class $[\chi_0]$ satisfies
\beq
 -n\frac { c_1(X)\cdot
[\chi_0]^{n-1}}{[\chi_0]^n}[\chi_0]+(n-1)c_1(X) \geq
0 \label{eq:100}
\eeq
then the $K$-energy is proper on $\cH(X, \chi_0).$
\end{cor}

Note that when the strict inequality holds, our condition (\ref{eq:100}) is stronger than
(\ref{eq:109}). For some technical
reasons, we cannot weaken (\ref{eq:100}) to Song-Weinkove's original
condition. Moreover, in Theorem \ref{theo:main1} we don't require the condition
that $c(X)<0$. In the case of $c_1(X)>0$, we have the following
result. Although  this result is very simple and its original proof
is very direct (see, for example,  Tian's book \cite{[Tian]}, page 95) , we feel it is still interesting to write it as a
corollary of Theorem \ref{theo:main1}.

\begin{cor}\label{cor:main3}Let $X$ be a compact K\"ahler manifold with $c_1(X)>0$.
If the $\alpha$-invariant $\al_X(\pi c_1(X))>\frac n{n+1}$, then
then the $K$-energy is proper on the K\"ahler class $\pi c_1(X).$
\end{cor}

An application of Corollary \ref{cor:main2} is the following result,
which says that the properness of the $K$-energy is not a sufficient
condition of the smooth convergence of the $J$-flow. This answers a
question of J. Ross \cite{[Ross]}. This result is also implied by
Corollary 1.2 of \cite{[FLSW]}.

\begin{cor}\label{cor:main4} There exists a compact K\"ahler surface
with a K\"ahler class $\Om$ such that the $K$-energy is proper on
$\Om$ but the $J$-flow doesn't converge smoothly.\\
\end{cor}

One might ask whether the conditions of Theorem \ref{theo:main1} are
optimal. We calculate two concrete examples here and apply Theorem
\ref{theo:main1} to determine
the K\"ahler classes on which the $K$ energy is proper.
 Let
$X$ be the blowup of $\CC \PP^2$ at three general points and  $E_1,\dots,E_3$
the exceptional divisors of the blowing up map. Denote by $F_i,\ i=1,2,3$ the strict transforms of lines through two of the three blowing up centers. Consider the class $L_{ \la}=(E_1+E_2+E_3)+\la\,
(F_1+F_2+F_3)$ for a positive rational number $\la$. Then applying
Theorem \ref{theo:main1}, if $\la$ satisfies
 \beq \frac 56<\la<\frac 65, \label{eq:102}\eeq
  then the $K$-energy  is proper on $\pi c_1(L_{ \la}).$ The details are contained in section \ref{sec:toric}.
This example was also studied by Zhou-Zhu in \cite{[ZhZh]}. They
analyzed the expression of the $K$-energy carefully on toric manifolds
 and showed that the $K$-energy is is proper on $G$-invariant
 metrics if
 \beq 0.61\approx {1\over 1+\frac{\sqrt 10}{5}}<\la < 1+\frac{\sqrt 10}{5}\approx 1.63. \label{eq:014}\eeq
We  see that the conditions (\ref{eq:014}) is less restrictive
than (\ref{eq:102}). Therefore, the conditions in Theorem
\ref{theo:main1} is not sharp.\\

Theorem \ref{theo:main1} shows that the $K$-energy and $\alpha$-invariants are closely related. In \cite{[Der]} Dervan proved that the $\al$ invariant also closely relates
the $K$-stability. In \cite{[Der]}, Dervan studied the $\al$-invariant and
$K$-stability for general polarizations on Fano manifolds. An example of
\cite{[Der]} is the the blowup of $\mathbb C \PP^2$
 at eight points in general positions. Let $E_i$ be the exceptional divisors,
 and $L_\lambda=3H-\sum_{i=1}^7E_i-\lambda E_8$,  where  $\lambda$
 is a  positive rational number. Dervan proved that when
\beq 0.76\approx\frac{1}{9}(10-\sqrt{10})<\lambda<\sqrt{10}-2\approx1.16,\label{eq:016}\eeq $(X,L_\lambda)$ is K-stable.
  Applying
Theorem \ref{theo:main1} to this example, we know that the $K$-energy is proper on $\pi c_1(L_{\la})$
if
\beq \frac{4}{5}<\lambda <\frac{10}{9}. \label{eq:015}\eeq
 Note that the interval (\ref{eq:015}) is strictly contained in (\ref{eq:016}).
According to Tian's conjecture and general
Yau-Tian-Donaldson conjecture, this example  hints that the
conditions in Theorem \ref{theo:main1} is not optimal.  There are some overlap regions between our
results and Dervan's for the properness of the $K$-energy and
$K$-stability.  All these results provide some support to general
Yau-Tian-Donaldson conjecture for general cscK metrics.\\

Finally, we reprove the existence of the critical point of $\hat J$
functional under the condition (\ref{eq:SW}) by the continuity
method. In \cite{[FLM]}, Fang-Lai-Ma discussed a class of fully nonlinear
flows in K\"ahler geometry, which includes the $J$-flow as a special
case. In Remark 1.3 of  \cite{[FLM]}, they asked
whether the critical points of those fully nonlinear
flows   can be solved by using the
elliptic method instead of the geometric flow method.  Later, in a series of papers
\cite{[GL1]}\cite{[GL2]}\cite{[GS]} Guan and his collaborators gave
some $C^2$ estimates for these critical points on Hermitian
manifolds. Then Sun proved Song-Weinkove's
result by the elliptic method on general Hermitian manifolds \cite{[S]}. Here we give a different proof only  using the
estimates in \cite{[SW]} \cite{[W1]} and \cite{[W2]}.

\begin{theo}(cf. \cite{[SW]})\label{theo:main4}
If there is a metric $\chi'\in [\chi_0]$ satisfying
  \beq
(nc\chi'-(n-1)\oo)\wedge \chi'^{n-2}>0, \nonumber
  \eeq
then there is a smooth K\"ahler metric $\chi\in [\chi_0]$
  satisfying the equation
  \beq \oo\wedge
\chi^{n-1}=c\chi^n. \nonumber\eeq

\end{theo}

In a forthcoming paper,
we will generalize the results in this paper
to the properness of the log $K$-energy and discuss the existence of critical points of
the $J$ flow with conical singularities. Our method can also be applied to
study the modified $K$-energy for extremal K\"ahler metrics and we
will discuss this elsewhere. \\

{\bf Acknowledgements}: The authors would like to thank Professor Julius Ross for
bringing our attention to this problem  and
 many helpful discussions. We are also grateful to the anonymous referee for valuable comments and suggestions. \\

\section{Preliminaries}
In this section, we recall some basic facts on $J$-flow. We follow
the notations in \cite{[SW]}. Let $(X, \oo)$ be a $n$-dimensional compact K\"ahler manifold with a
K\"ahler form $\oo=\frac{\sqrt -1}{2}g_{i\bar j}dz^i\wedge d\bar z^j$, and $\chi_0$ another K\"ahler form on $X.$ We
denote by $\cH(X, \chi_0)$ the space of K\"ahler potentials
$$\cH(X, \chi_0)=\{\varphi\in C^{\infty}(X, \RR)\;|\; \chi_{\varphi}=\chi_0+\pbp \varphi>0\}.$$
The $J$-flow is defined by \beq \pd {\varphi}t=c-\frac {\oo\wedge
\chi_{\varphi}^{n-1}}{\chi_{\varphi}^n},\quad
\varphi|_{t=0}=\varphi_0\in \cH, \label{eq:001}\eeq where $c$ is a
constant defined by \beq c=\frac {[\oo]\cdot
[\chi_0]^{n-1}}{[\chi_0]^n}. \label{eq:c}\eeq A critical point of
$J$-flow is a K\"ahler form $\chi$ satisfying \beq \oo\wedge
\chi^{n-1}=c\chi^n. \label{eq:002} \eeq Donaldson \cite{[Don1]}
showed that a necessary condition for the existence of the critical
metrics (\ref{eq:002}) is $[c\chi_0-\oo]>0$, and Chen \cite{[Chen1]}
showed that it is also sufficient in complex dimension $2$. In a
series of papers \cite{[W1]}\cite{[SW]} Weinkove and Song-Weinkove
obtain a sufficient and necessary condition for any dimension:

\begin{theo}\label{theo:SW}(cf. \cite{[SW]})The following conditions are equivalent:
\begin{enumerate}
  \item[(1).] There is a metric $\chi'\in [\chi_0]$ satisfying
  \beq
(nc\chi'-(n-1)\oo)\wedge \chi'^{n-2}>0. \label{eq:SW}
  \eeq
  \item[(2).] For any initial data $\varphi_0\in \cH$, the $J$-flow
  (\ref{eq:001}) converges smoothly to $\varphi_{\infty}\in \cH$
  with the limit metric $\chi_{\infty}$ satisfying (\ref{eq:002}).
  \item[(3).] There is a smooth K\"ahler metric $\chi\in [\chi_0]$
  satisfying the equation (\ref{eq:002}).
\end{enumerate}

\end{theo}

The convergence of the $J$-flow can be used to determine in which
K\"ahler class the $K$-energy is proper or bounded from below.
Note that the $J$-flow is the gradient flow of the functional
\beq
\hat J_{\oo, \chi_0}(\varphi)=\int_0^1\,\int_X\;\pd {\varphi}t
(\oo\wedge \chi_{\varphi}^{n-1}-c\chi_{\varphi}^n)\frac
{dt}{(n-1)!}. \label{eq:003}
\eeq
When $\oo$ is a positive $(1, 1)$ form, the functional $\hat J_{\oo,
\chi_0}$ is strictly convex along any $C^{1, 1}$ geodesics(cf.
\cite{[Chen1]}). Therefore, under the assumption (1) of Theorem
\ref{theo:SW}, the critical point of $\hat J_{\oo,
\chi_0}$ exists and  $\hat J_{\oo,
\chi_0}$ is bounded from below.
When $c_1(X)<0$, we can choose   $\oo=-Ric(\chi_0)>0$, then the
 $K$-energy can be written as
\beq
\mu_{\chi_0}(\varphi)=\int_X\;\log \frac
{\chi_{\varphi}^n}{\chi_0^n}\,\frac {\chi_{\varphi}^n}{n!}+\hat J_{\oo,
\chi_0}(\varphi).
\eeq where $\oo=-Ric(\chi_0)>0$. Since the first term of this formula is always proper (see Lemma 4.1 of \cite{[SW]}), one can conclude that the $K$-energy is proper on $[\chi_0]$, provided (\ref{eq:SW}) holds for $[\oo]=-\pi c_1(X)$.\\

\section{Proofs of Theorem \ref{theo:main1} and Corollary \ref{cor:main2}-\ref{cor:main4} }
In this section, we prove  Theorem \ref{theo:main1} and Corollary
\ref{cor:main2}-\ref{cor:main4}.

\begin{proof}[Proof of Theorem \ref{theo:main1}]
We focus on the $G=\{1\}$ case, the proof in the general case is
identical. Recall the Aubin-Yau functionals \beqs
I_{\chi_0}(\varphi)&=&\int_X\;\varphi(\frac {\chi_0^n}{n!}-\frac {\chi_{\varphi}^n}{n!}),\\
J_{\chi_0}(\varphi)&=&\int_0^1dt\int_X\;\pd
{\varphi_t}t(\frac {\chi_0^n}{n!}-\frac {\chi_{\varphi}^n}{n!}).
\eeqs Direct calculation shows that
\beqs
I_{\chi_0}(\varphi)-J_{\chi_0}(\varphi)&=&-\int_0^1\,dt\int_X\;\pd
{\varphi}t \Delta_{\chi_{\varphi}}\varphi\,\frac {\chi_{\varphi}^n}{n!}\\
&=&-\int_0^1\,\int_X\;\pd {\varphi}t (\chi_{\varphi}^n-\chi_0\wedge
\chi_{\varphi}^{n-1})\frac {dt}{(n-1)!}. \eeqs
By the assumption (2), there exists $\chi_1\in [\chi_0]$
such that $\oo_1:=-Ric(\chi_1)$ satisfies
\beq \oo_1+\ee\chi_0>0.\label{eq:106}\eeq
By the assumption (3), there exists $\td  \chi', \td
\chi \in [\chi_0] $ such that
\beq
(nc+\ee)\td \chi'-(n-1)\td \oo>0, \nonumber
\eeq where $\td \oo:=-Ric(\td \chi)$ and
\beq
c=\frac{-\pi c_1(X)\cdot [\chi_0]^{n-1}}{[\chi_0]^n}.\nonumber
\eeq
Note that $c+\ee>0$ by assumption (2).
 Set
$$\chi':=\frac 1{n(c+\ee)}\Big((nc+\ee)\td \chi'+(n-1)\ee \chi_0+(n-1)\pbp \log
\frac { \chi_1^n}{\td\chi^n}\Big).$$
Then $\chi'\in [\chi_0]$ and $\chi'$ satisfies
\beq
n(c+\ee)\chi' -(n-1)(\oo_1+\ee\chi_0)=(nc+\ee)\td \chi'-(n-1)\td \oo>0.
\label{eq:108}
\eeq
In particular, $\chi'>0$.
We define the
modified $\hat J$ functional by \beqs \hat J^{\ee}_{\oo_1, \chi_0
}(\varphi)&=&\hat J_{\oo_1, \chi_0 }(\varphi)+
\ee \Big(I_{\chi_0}(\varphi)-J_{\chi_0}(\varphi)\Big)\\
&=&\int_0^1\,\int_X\;\pd {\varphi}t\Big((\oo_1+\ee \chi_0)\wedge \chi_{\varphi}^{n-1}
-(c+\ee)\chi_{\varphi}^n\Big)\frac {dt}{(n-1)!},
\eeqs which is exactly the functional $\hat J_{\oo_1+\ee \chi_0,
\chi_0}$ defined by (\ref{eq:003}). Thus, by Chen's result in \cite{[Chen1]}
if   there is a K\"ahler
metric $\chi$ satisfying
\beq
(\oo_1+\ee\chi_0)\wedge \chi^{n-1}=(c+\ee)\chi^n, \label{eq:004}
\eeq then $\hat J^{\ee}_{\oo_1, \chi_0 }$ is bounded from below on
$\cH(X, \chi_0).$ By Theorem \ref{theo:SW} the critical metric (\ref{eq:004})
exists if there exists a K\"ahler metric $\chi' \in [\chi_0]$ such
that
\beq
\Big(n(c+\ee)\chi' -(n-1)(\oo_1+\ee\chi_0)\Big)\wedge \chi'^{\,n-2}>0. \label{eq:009}
\eeq
Clearly, (\ref{eq:009}) can be implied by (\ref{eq:108}).
Therefore, if (\ref{eq:106}) and (\ref{eq:108}) hold, then $\hat J^{\ee}_{\oo_1, \chi_0 }$ is bounded from below on
$\cH(X, \chi_0)$ and we have
\beq  \hat J_{\oo_1, \chi_0 }(\varphi)\geq -
\ee \Big(I_{\chi_0}(\varphi)-J_{\chi_0}(\varphi)\Big)-C,\quad \forall\;\varphi\in
\cH(X, \chi_0). \label{eq:011}\eeq\\

Next, we claim that for $\oo:=-Ric(\chi_0)$, there is a constant $C(
\chi_1, \chi_0)$ such that for any $\varphi\in \cH(X, \chi_0)$
\beq
|\hat J_{\oo, \chi_0}(\varphi)-\hat J_{\oo_1, \chi_0}(\varphi)|\leq
C( \chi_1, \chi_0). \label{eq:107}
\eeq
In fact, by the explicit expression of the $\hat J$ functional from
\cite{[Chen1]} we have
\beqs
\hat J_{\oo, \chi_0}(\varphi)&=&\sum_{p=0}^{n-1}\,\frac
1{(p+1)!(n-p-1)!}\int_X\;\varphi\,\oo\wedge \chi_0^{n-p-1}\wedge
(\pbps \varphi)^{p}-nc\int_0^1\,dt\int_X\;\pd {\varphi_t}t\,\frac {\chi_{\varphi_t}^n}{n!}\\
&=&\sum_{p=0}^{n-1}\,c_p\int_X\; \varphi \,\oo\wedge\chi_0^{n-p-1}\wedge
(\chi_{\varphi}-\chi_0)^{p}-nc\int_0^1\,dt\int_X\;\pd {\varphi_t}t\,\frac {\chi_{\varphi_t}^n}{n!}\\
&=&\sum_{p=0}^{n-1}\,c_p'\int_X\; \varphi \,\oo\wedge\chi_0^{n-p-1}\wedge
\chi_{\varphi}^{p}-nc\int_0^1\,dt\int_X\;\pd {\varphi_t}t\,\frac {\chi_{\varphi_t}^n}{n!},
\eeqs where $c_p, c_p'$ are universal constants. Since $\oo_1-\oo=-\pbp
f$ where $f=\log\frac {\chi_{1}^n}{\chi_0^n}$, we have
\beqs
 |\hat J_{\oo, \chi_0}(\varphi)-\hat J_{\oo_1,
 \chi_0}(\varphi)|&\leq&
 \Big|\sum_{p=0}^{n-1}\,\frac {c_p'}2\int_X\; \varphi \,\pbps f\wedge\chi_0^{n-p-1}\wedge
\chi_{\varphi}^{p}\Big|\\
&=&\Big|\sum_{p=0}^{n-1}\, c_p' \int_X\; f \,(\chi_{\varphi}-\chi_0)\wedge\chi_0^{n-p-1}\wedge
\chi_{\varphi}^{p}\Big|\leq C\|f\|_{C^0}.
\eeqs
Therefore, (\ref{eq:107}) is proved. \\

Now using Tian's $\al$-invariant we have (see Lemma 4.1 of \cite{[SW]} )
\beqn\int_X\;\log \frac
{\chi_{\varphi}^n}{\chi_0^n}\,\frac {\chi_{\varphi}^n}{n!}&\geq&
 \al  I_{\chi_0}(\varphi)-C\nonumber\\&\geq& \frac {n+1}{n}\al \cdot (I_{\chi_0}(\varphi)
-J_{\chi_0}(\varphi))-C,\quad \forall\,\varphi\in \cH
\label{eq:012}\eeqn
for any $\al\in (0, \al_{[\chi_0]}(X)).$ Combining the inequalities
(\ref{eq:011})-(\ref{eq:012}) we have
\beqs
\mu_{\chi_0}(\varphi)&=&\int_X\;\log \frac
{\chi_{\varphi}^n}{\chi_0^n}\,\frac {\chi_{\varphi}^n}{n!}+\hat J_{\oo, \chi_0}(\varphi)\\
&\geq &\int_X\;\log \frac
{\chi_{\varphi}^n}{\chi_0^n}\,\frac {\chi_{\varphi}^n}{n!}+\hat J_{\oo_1,
\chi_0}(\varphi)-C(\chi_0, \chi_1)\\&\geq &
 \Big(\frac
{n+1}{n}\al-\ee\Big)\Big(I_{\chi_0}(\varphi)-J_{\chi_0}(\varphi)\Big)-C.\eeqs
Therefore, if $\al_X([\chi_0])>\frac n{n+1}\ee$ then the $K$ energy
is proper.

\end{proof}

\begin{rem}We see from the above proof that the $K$-energy is proper
if (\ref{eq:106}) and (\ref{eq:009}) hold. However, we are unable to
show that (\ref{eq:009}) holds if there exists a K\"ahler metric $\chi'\in [\chi_0]$
such that
 \beq
\Big(-n\frac {\pi c_1(X)\cdot
[\chi_0]^{n-1}}{[\chi_0]^n}\chi'-(n-1)\oo+\ee \chi'\Big)\wedge
\chi'^{\,n-2}>0.
\nonumber
\eeq If it were true, then Song-Weinkove's question would be
answered as a corollary.\\

\end{rem}

\begin{proof}[Proof of Corollary \ref{cor:main2}] Let $\ee\in (0, \frac {n+1}{n}\al_X([\chi_0])).$
Since the manifold $X$ has negative first Chern class, the
condition  $(2)$ in  Theorem \ref{theo:main1}
automatically holds. The third condition in  Theorem \ref{theo:main1}
follows directly from (\ref{eq:100}). Thus, the corollary is
proved.

\end{proof}

\begin{proof}[Proof of Corollary \ref{cor:main3}] Let $[\chi_0]=\pi c_1(X)$. By the assumption
 $\al_X([\chi_0])>\frac n{n+1}$,
we can  choose $\ee\in (1, \frac {n+1}{n}\al_X([\chi_0]))$.
Therefore, the three conditions of  Theorem \ref{theo:main1} hold
and the corollary is proved.

\end{proof}

\begin{proof}[Proof of Corollary \ref{cor:main4}]
In complex dimension 2, the $J$ flow converges smoothly if and only
if the inequality
\beq -2\frac {c_1(X)\cdot
[\chi_0]^{n-1}}{[\chi_0]^n}\chi'+c_1(X)>0 \label{eq:101}\eeq
holds for some $\chi'\in [\chi_0]$. By Corollary \ref{cor:main3},
the $K$-energy is still proper on $\cH(X, \chi_0)$ if the inequality
(\ref{eq:101}) is not strictly. Therefore, on any K\"ahler class
lying on the boundary of the cone defined by  (\ref{eq:101}), the
$K$-energy is proper but the $J$-flow doesn't converge smoothly.

\end{proof}

\section{Toric varieties}\label{sec:toric}
In this section, we apply Theorem \ref{theo:main1} to  projective toric
manifolds. First we recall some basic facts on toric varieties from Fulton's book \cite{[Ful]}. Let $N$ be a lattice of rank $n$, $M=Hom(N,\mathbb{Z})$ is its dual. The complex torus group $T$ is defined to be $T=N\otimes_\ZZ \mathbb{C}^*=Hom(M,\mathbb{C}^*)$, and we have the real torus group $T_R=N\otimes_\ZZ\mathbb{S}^1=Hom(M,\mathbb{S}^1)$. A complete toric variety $X_\Delta$ is defined by a  fan $\Delta$ in $N_\RR=N\otimes_\ZZ\RR$, consists of strongly convex rational polyhedral cones, such that the union of these cones is the whole of $N_\RR$. We assume each cone is generated by a subset of a $\ZZ$-basis of $N$, then $X_\Delta$ is smooth.
Each 1-dimensional cone $\rho_i$ of $\Delta$ corresponds to an irreducible divisor $D_i$ and it is well known that the Picard group of $X_\Delta$ is generated by these $D_i$'s.

Let $u_i\in N$ be the primitive generator of the 1-dimensional cone $\rho_i$. For any $T$-invariant divisor $D=\sum_i a_iD_i$, $a_i\in\ZZ$, we associate to it a piecewise linear function $\phi_D$ on $N_\RR$, defined by
$$\phi_D(u_i)=-a_i.$$
The divisor $D$ is ample if and only if $\phi_D$ is strictly concave \footnote{In Fulton's terminology, it is called ``convex". We choose this name according to the usual notation.} in the sense that for any $n$-dimensional cone $\sigma$ of $\Delta$, the graph of $\phi_D$ on the compliment of $\sigma$ is strictly under the graph of the linear function $u_\sigma$ whose restriction on $\sigma$ equals $\phi_D$.
For such a $D$, we can also associate to it a polytope $P_D$ in $M_\RR$, defined by
$$P_D:=\{m\in M_\RR\big| \langle m, v\rangle \geq \phi_D(v), \forall v\in N_\RR\}.$$
We have
$$H^0(X_\Delta,\mathcal{O}(D))\cong \bigoplus_{m\in P_D\cap M}\CC \chi^m,$$
where $\chi^m$ is the rational function on $X_\Delta$ defined by $m\in M$. If we restrict  $\chi^m$ to $(\CC^*)^n$, the $\chi^m$ has the form $z_1^{m_1}\cdots z_n^{m_n}$.

Besides the natural $T$-action, the toric manifold $X_\Delta$ also has some discrete symmetries from the fan. Let
\[\mathcal W= \{g\in GL(n,\ZZ) \big |\  g\  \text{ preserves}\ \Delta \}.\]
Since each $g\in \mathcal W$ is decided by a permutation of  $\{u_i\}$, so $\mathcal W$ is finite. Every $g$ induces a $\bar{g}\in Aut(X)$.   For a given ample divisor $D$, we consider the subgroup of $\mathcal W$ preserving the class of $D$:
$$\mathcal K_D:=\{g\in \mathcal W \big | \bar g^*D\sim D\},$$
where ``$\sim$" means ``linearly equivalent". Note that $\bar g^*D$ corresponds to the function $\phi_D\circ g$, so we have a combinatorial characterization of $\mathcal K_D$:
$$\mathcal K_D=\{g\in \mathcal W \big | \exists m\in M,\ s.t.\ \phi_D\circ g=\phi_D+\langle m, \cdot\rangle\}.$$

Let $G$ be the compact subgroup of $Aut(X)$ generated by $T_\RR$ and $\mathcal K_D$, we compute the $\alpha$-invariant $ \alpha_G(\pi [D])$, extending previous works of Song \cite{[So]} and Cheltsov-Shramov \cite{[ChSh]} in the toric Fano case. 

First, we make a normalization: assume the barycenter of $P_D$ is the origin of $M_{\RR}$. This is equivalent to taking power of the line bundle $\mathcal{O}(D)$ and change the divisor in its linear equivalent class. Since we are computing the $\alpha$-invariant, this does note lose any generality.
Under this assumption, for any $g\in \mathcal K_D$, the induced linear action $g^*$ on $M_\RR$ preserves $P_D$. Actually, by definition of $\mathcal K_D$, $g^*P_D$ is a translation of $P_D$. However, since $g^*$ is linear, the barycenter of $g^*P_D$ is also the origin. This implies $g^*P_D=P_D$. Now the result is:

\begin{theo}\label{theo:toric}The $\alpha_G$ invariant equals
$$\alpha_G(\pi [D])=\underset{i}{\min}\underset {y\in P_D^{\mathcal K}}{\min}\frac{1}{\langle  y,u_i\rangle + a_i},$$
where $P_D^{\mathcal K}$ is the set of  $\mathcal K_D$ (induced action on $M_\RR$) fixed points in $P_D$ .
\end{theo}

\begin{proof}
To prove this, we use the result of J.-P. Demailly \cite{[Dem]}, saying that

$$\alpha_G(\pi [D])=\underset{k\in \mathbb{Z}^+}{\inf}\underset{|\Sigma|\subset|kD|}{\inf}lct(\frac{1}{k}|\Sigma|)$$
where the second infimum is taken for all $G$ invariant linear systems. If $\Sigma_1\subset\Sigma_2\subset|kD|$, obviously $lct(\frac{1}{k}|\Sigma_1|)\leq lct(\frac{1}{k}|\Sigma_2|)$, so we only need take all $G$ invariant and irreducible linear systems $|\Sigma|$.

Note that $H^0(X,kD)=span\{s_m\ | \ m \in kP_D\cap M\}$ is just the decomposition into one dimensional invariant subspaces of the $T$-action, and  the torus acts on these lines with different characters. Take a $G$ invariant linear system $\Sigma\subset H^0(X,kD)$, $\Sigma$ must be spanned by $\{s_m, m \in \Gamma \}$ for some $\Gamma \subset kP_D\cap M$, and  $\Gamma $ is $\mathcal K_D$ invariant. We can assume $\Sigma $ is irreducible, so $\Gamma$ is an orbit of the $\mathcal K_D$ action. Denoted $\#\Gamma =N$, then $s=\underset{m \in \Gamma}{\otimes}s_m \in H^0(X,kND)$, and it spans a one dimensional $G$-invariant linear system. Moreover, we have $lct(\frac{1}{kN}(s))\leq lct(\frac{1}{k}|\Sigma|)$ by H\"older inequality. So without loss of generality, we can take $\Sigma$ to be one dimensional in the following.

Let $s \in H^0(X,kD)$. Assume $s$ corresponds to the lattice point $m \in kP_D\cap M$, and $m$ is fixed by $\mathcal K_D$ . Now the  divisor of $s$ is given by
$$(s)=\sum_i (\langle m, u_i \rangle + ka_i) D_i.$$
Since $X$ is smooth , $D_i$'s have simple normal crossing intersections with each other. So we have $$lct(\frac{1}{k}(s))= \underset{i}{\min}\frac{k}{\langle m, u_i \rangle + ka_i}.$$ From this and the above discussion, we have
\[\alpha_G(\pi[D])=\underset{k\in \mathbb{Z}^+}{\inf}\underset{\frac{m}{k} \in P_D\cap \frac{1}{k}M}{\inf}\underset{i}{\min}\frac{1}{\langle \frac{m}{k},u_i \rangle + a_i}=\underset{i}{\min}\underset {y\in P_D^{\mathcal K}}{\min}\frac{1}{\langle y,u_i\rangle + a_i}.\]
\end{proof}

Now we consider a concrete example: the blowup of $\CC P^2$ at three general points. The fan of $M$ is generated by the following primitive vectors:
$$u_1=(1,0),\ u_2=(1,1),\ u_3=(0,1),\ u_4=(0,-1),\ u_5=(-1,-1),\ u_6=(0,-1).$$
They correspond to all the $(-1)$-curves $D_1,\dots,D_6$ on $M$.  The intersection numbers satisfy :
$$D_i\cdot D_i=-1,\ D_i\cdot D_{i+1}=1, i=1,\dots,6,$$
where $D_7=D_1$ is understood. Note also that the anti-canonical divisor is given by $K^{-1}=D_1+\cdots+D_6$.
The bundle we choose is $L_\la=(D_1+D_3+D_5)+\la(D_2+D_4+D_6)$, where $\la\in\QQ$. This is the same class as we mentioned in the introduction. Actually if we blow down $D_1, D_3$ and $D_5$, we will get $\CC P^2$, and the image of $D_2, D_4, D_6$ are lines.
\begin{prop}Under the above assumptions, if $\frac{5}{6}<\la<\frac{6}{5}$ then the $K$-energy is
proper on the space of $G$-invariant potentials for the class $\pi c_1(L_{\la})$.
\end{prop}
\begin{proof}
 For this class to be ample, we need the function $\phi_{L_\la}$ to be strictly concave, and this is easily seen to be
 $${1\over{2}}<\la<2.$$
Now we apply our theorem to this case.  Write $D=aL_\la$ for some positive $a$. The $\alpha$-invariant can be computed using our Theorem \ref{theo:toric}: Since the elements of $\mathcal K_D$ can be enumerated, one can check easily that $P_D^\mathcal K=\{0\in M_\RR\}$. So
$$\alpha_G(\pi[D])=\min\{\frac{1}{a},\frac{1}{\la a}\}.$$
 So the condition $\alpha_G>{2\over 3}$ translates to $0<a<{3\over 2}, 0<a<{3\over 2\la}$.
 Similarly, by considering the concexity of $\phi_{K+D}$, we can translate the condition $K+D>0$ into
 $$(2-\la)a>1,\quad (2\la-1)a>1.$$
 The last condition says that
$$(1-\underline R)D+K^{-1}>0,$$
where $\underline R$ is the mean value of the scalar curvature of the class $\pi [D]$.
Direct computation\footnote{Just compute the intersection number, or use the fact that $\underline R={Vol(\partial P_L)\over Vol(P_L)}$, where $Vol(\partial P_L)$ is computed using Donaldson's special boundary measure, see \cite{[Don2]} and\cite{[ZhZh]} .} shows that
$$\underline R=\frac{2(1+\la)}{a\la(4\la-1-\la^2)}.$$
Again, use the piecewise linear function $\phi_{(1-\underline R)D-K}$, we get the condition
$${2(\la+1)\over 4\la-1-\la^2}+{1\over \la-2}<a,\quad {2(\la+1)\over 4\la-1-\la^2}+{1\over 1-2\la}<a.$$
In conclusion, we have
$${1\over 2}<\la<2,$$
and
\beqs
\max\Big\{ \frac{1}{2-\la},\frac{1}{2\la-1},{2(\la+1)\over 4\la-1-\la^2}+{1\over \la-2},{2(\la+1)\over 4\la-1-\la^2}+{1\over 1-2\la} \Big\}<a\\<\min \{\frac{3}{2},\frac{3}{2\la}\}.
\eeqs
Then for any $\frac{5}{6}<\la<\frac{6}{5}$, we can find a suitable $a$ satisfying these conditions. So we can apply Theorem \ref{theo:main1} (or Corollary \ref{cor:main1} ) to conclude that the $K$-energy is
proper on the space of $G$-invariant potentials for the class $\pi c_1(aL_{\la})$ and hence on $\pi c_1(L_{\la})$. \\
\end{proof}

\section{An example of Dervan}
In this section, we would like to compare our result with that of Dervan \cite{[Der]} on the $\alpha$-invariant and K-stability for general polarizations on Fano manifolds. His sufficient condition involves the quantity
$$\mu(X,L):=\frac{-K_X\cdot L^{n-1}}{L^n}.$$
And he proved that if $\alpha_X(L)>\frac{n}{n+1}\mu(X,L)$ and $-K_X\geq \frac{n}{n+1}\mu(X,L)L$, then $(X,L)$ is K-stable. This condition is quite similar to ours, but at present we don't know whose condition is stronger.

We know compare our result with that of Dervan on a concrete example, the Del Pezzo surface $X$ of index one. It is the blowup of $\mathbb CP^2$ at eight points in general positions. Let $E_i$ be the exceptional divisors, and $L_\lambda=3H-\sum_{i=1}^7E_i-\lambda E_8$, then $L_1=-K_X$ and we set $L=aL_\lambda$,  where $a$ and $\lambda$ are both positive and $\la\in\QQ$. We know the exceptional curves on $X$ are $E_i$'s and the strictly transforms of following curves in $\mathbb CP^2$:
\begin{itemize}
\item lines through 2 points,
\item conics through 5 points,
\item cubics through 7 points, vanishing doubly at 1 of them,
\item quartics through 8 points, vanishing doubly at 3 points,
\item quintics through 8 points, vanishing doubly at 6 points,
\item sextics through 8 points, vanishing doubly at 7 points and triply at another point.
\end{itemize}

\begin{prop}Under the above assumptions, if
$\frac{4}{5}<\lambda <\frac{10}{9}$ the $K$-energy is proper on the
 K\"ahler class $\pi c_1(L_{\la})$.

\end{prop}
\begin{proof}
By Kleiman's ampleness criterion, we know that $L_\lambda$ is ample when $\lambda < \frac{4}{3}$. According to Dervan\cite{[Der]}, we have
\[\alpha (\pi [L_\lambda])\geq \min \{1,\frac{1}{2-\lambda}\}. \]
Now we look for the $\lambda$ for which the class $L=aL_\la$ satisfies our conditions for some $a>0$.
The condition on the $\alpha$-invariant gives $$\frac{1}{a}\min \{1,\frac{1}{2-\lambda}\} > \frac{2}{3}.$$
For the ampleness of $L+K$, by computing the intersection number with the above exceptional curves we have
\[a > \frac{1}{4-3\lambda}, \text{ when} \ \lambda \geq 1 ,\]
 and\[ \ a >\frac{1}{\lambda},\text{ when}\ \lambda<1. \]
For the last condition, denote $b=\frac{4-2\lambda}{2-\lambda^2}-a$, we require the divisor $3(1-b)H-\sum_{i=1}^7(1-b)E_i-(1-\lambda b)E_8$ to be ample. This is equivalent to the conditions
\[b < \frac{1}{\lambda}, \text{ when} \ \lambda \geq 1 ,\]
 and \[ b <\frac{1}{4-3\lambda},\text{ when}\ \lambda<1. \]
From these inequalities, now we can conclude that when
 $\frac{4}{5}<\lambda <\frac{10}{9}$, the K-energy is proper on the
 K\"ahler class $3H-\sum_{i=1}^7E_i-\lambda E_8$.\\
\end{proof}

\section{Existence of critical points  by the continuity method }
In this section, we prove Theorem \ref{theo:main4} by the continuity method by
using the estimates of \cite{[SW]}\cite{[W1]} and \cite{[W2]}.

\begin{proof}[Proof of Theorem \ref{theo:main4}]
By the assumption,  we can
choose the reference metric $\chi_0$ as $\chi'$ in (\ref{eq:SW}).
Consider the continuity equation \beq \oo_t\wedge
\chi_t^{n-1}=c_t\chi_t^n, \label{eq:continuity}\eeq where
$$\chi_t=\chi_0+\pbp \varphi_t,\quad \oo_t=(1-t)\chi_0+t\oo. $$
Clearly, $\oo_0=\chi_0$ and $\oo_1=\oo. $ The constant $c_t$ is
given by \beq c_t=\frac {[\oo_t]\cdot
[\chi_0]^{n-1}}{[\chi_0]^n}=(1-t)+c\cdot t, \label{eq:ct}\eeq where
the constant $c$ is given by (\ref{eq:c}).  We denote by $g_t$ the
corresponding metric tensor of the K\"ahler form $\oo_t$. Clearly,
$\varphi=0$ is a solution to the equation (\ref{eq:continuity}) when
$t=0.$ Define the set
$$\cS=\{s\in [0, 1]\;|\; \hbox{the equation (\ref{eq:continuity}) has a solution
 when $t=s$} \}.$$
 
 We first prove the openness of $\cS$. Consider the operator $L_t(\varphi)=\chi_\varphi^{i\bar j}g_{t, i\bar j}: \cU\to C^{\alpha}_0$, where
 $$\cU:=\{\varphi\in C^{2,\alpha}(X;\RR)| \chi_0+\pbp\varphi>0\}/\RR,$$
 and $$C^{\alpha}_0:= C^{\alpha}(X;\RR)/\RR.$$ 
 Here we take the quotient space norm, and the H\"older semi-norms are taken with respect to a fixed metric, say, $\chi_0$.
 For any $t_0\in \cS$,
  the linearization of the operator
$L_t(\varphi)$ at $(t_0,
\varphi_{t_0})$ , $DL|_{(t_0, \varphi_{t_0})}: C^{2,\alpha}_0\to C^\alpha_0$ is given by
$$DL|_{(t_0, \varphi_{t_0})}(f)=-h_{t_0}^{p\bar q}f_{p\bar q},
\quad h_t^{p\bar q}=\chi_t^{i\bar q}\chi_t^{p\bar j}g_{t, i\bar
j}.$$ Here $C^{2,\alpha}_0:=C^{2,\alpha}(X;\RR)/\RR$. It is obvious to see that  $DL|_{(t_0, \varphi_{t_0})} $ (We write $DL$ for short.) is elliptic. By strong maximum principle, its
kernel in $C^{2,\alpha}_0$ is trivial. We claim that $DL$ is a self-adjoint operator in $L^2$. Actually, for any real-valued smooth function $\eta$, we have 
\footnote{We can also prove this by observing that $(DL(f),\eta)_{L^2}=-\int_X \eta <\pbp f, \oo_t>_{\chi_{t_0}} {\chi_{t_0}^n\over n!}=
 -\int_X \eta \pbp f\wedge \star \oo_t$. Here $\star$ is the Hodge star operator associated with $\chi_{t_0}$. Since $\partial $ and $\bar\partial$ commute with $\star$ It is obvious that this equals $-\int_X f \pbp \eta\wedge \star \oo_t= (f,DL(\eta))_{L^2}$. A good reference for Hodge star operator on K\"ahler manifold is \cite{[Wells]}. }
\beqs
(DL(f),\eta)_{L^2}&=& \int_X \eta DL(f) {\chi_{t_0}^n\over n!}=-\int_X \eta \chi_{t_0}^{i\bar q}\chi_{t_0}^{p\bar j}g_{t_0, i\bar
j} f_{p\bar q}  {\chi_{t_0}^n\over n!}\\
&=& \int_X \eta_p \chi_{t_0}^{i\bar q}\chi_{t_0}^{p\bar j}g_{t_0, i\bar
j} f_{\bar q}  {\chi_{t_0}^n\over n!}+\int_X \eta \chi_{t_0}^{i\bar q}\chi_{t_0}^{p\bar j}g_{t_0, i\bar
j, p} f_{\bar q}  {\chi_{t_0}^n\over n!}
\eeqs
Here all the covariant derivatives are taken with respect  to $\chi_{t_0}$. 
Observe that $\chi_{t_0}^{i\bar j}g_{t_0, i\bar
j}=c_{t_0}$, so $\chi_{t_0}^{i\bar j}g_{t_0, i\bar
j,p}=0$. Since both $\oo_{t_0}$ and $\chi_{t_0}$ are K\"ahler, we have $g_{t_0, i\bar
j,p}=g_{t_0, p\bar j,i}$ by comparing the expressions in local coordinates. So $\chi_{t_0}^{p\bar j}g_{t_0, i\bar
j, p}=\chi_{t_0}^{p\bar j}g_{t_0, p\bar
j, i}=0$. We have
$$(DL(f),\eta)_{L^2}=\int_X \eta_p \chi_{t_0}^{i\bar q}\chi_{t_0}^{p\bar j}g_{t_0, i\bar
j} f_{\bar q}  {\chi_{t_0}^n\over n!}.$$
Then another integration by parts gives the result: $(DL(f),\eta)_{L^2}=(f,DL(\eta))_{L^2}$.

 Now we conclude that the kernel of the adjoint of $DL$ is also trivial (modulo constant functions), and hence $DL|_{(t_0,\varphi_{t_0})}$ is invertible. By standard inverse function theorem, the set  $\cS$ is open. 
 
 Since $0\in \cS$, by the
  openness there exists $t_1>0$ such that (\ref{eq:continuity}) has
  a smooth solution for $t\in [0, t_1).$ By Lemma \ref{lem:2order}
  and Lemma \ref{lem:zero} below,  for any $t\in \cS$ the metric $\chi_t$ is uniformly equivalent
  to $\chi_0$. By the elliptic estimates in \cite{[GT]} we can get $C^{2, \al}(X, \oo)$ estimates of
   $\varphi_t$. \footnote{The adaptation of the classical Evans-Krylov theorem to the complex case has been carried out by Siu Y.-T. in \cite{[Siu]} P100-107. }
    This together with the classical Schauder
  estimates can improve the regularity to $C^{\infty}.$ Thus, $\cS$
  is closed and the   theorem is proved.

\end{proof}

\begin{lem}\label{lem:2order}Fix some $t_2\in (0, t_1).$ Let $\varphi_t$ be
a solution of (\ref{eq:continuity}) with $t\geq t_2>0$.
If the inequality (\ref{eq:SW}) holds, then there exist
constants $A, \la,  C>0$ independent of $t$ such that
$$\La_{\oo}\chi_t\leq \la \La_{\oo_t}\chi_t\leq \la C  e^{A(\varphi_t-\inf_X \varphi_t)},
\quad \forall\; t\geq t_2.$$

\end{lem}
\begin{proof}
Here we follow the estimates in \cite{[SW]} and \cite{[W1]}. Sometimes we omit the
subscript $t$ for simplicity. We choose $\chi_0$ to be the $\chi'$ in (\ref{eq:SW}). Write
$$\td \Delta f=\frac 1n\chi^{k\bar j}\chi^{i\bar l}g_{t, i\bar j}\p_k\p_{\bar l}f=
\frac 1nh^{k\bar l}\p_k\p_{\bar l}f.$$ We calculate \beq \td \Delta
(\La_{\oo_t}\chi)=\frac 1n h^{k\bar l}R_{k\bar l} ^{\quad  i\bar j}
(g_t)\chi_{i\bar j}+\frac 1n h^{k\bar l}g_t^{i\bar j}\p_k\p_{\bar
l}\chi_{i\bar j}, \eeq where $R_{k\bar l} ^{\quad  i\bar j}
(g_t)$ denotes the curvature tensor of $g_t$.  Note that by equation
(\ref{eq:continuity}),\beqs
0&=&-g_t^{i\bar j}\p_i\p_{\bar j}(\chi^{k\bar l}g_{t, k\bar l})\\
&=&g_t^{i\bar j}h^{p\bar q}\p_i\p_{\bar j}\chi_{p\bar q}-g_t^{i\bar
j}h^{r\bar q}\chi^{p\bar s}\p_i \chi_{r\bar s}\p_{\bar
j}\chi_{p\bar q}-g_t^{i\bar j}h^{p\bar s}\chi^{r\bar q}\p_i
\chi_{r\bar s}\p_{\bar j}\chi_{p\bar q}+\chi^{k\bar l}R_{k\bar
l}(g_t). \eeqs
 Therefore, we have
\beqs \td \Delta \log (\La_{\oo_t}\chi)&=&\frac {\td \Delta
(\La_{\oo_t}\chi)}{\La_{\oo_t}\chi}
-\frac {|\td \Na(\La_{\oo_t}\chi) |^2}{(\La_{\oo_t}\chi)^2}\\
&= &\frac 1{n\La_{\oo_t}\chi}\Big(h^{k\bar l}R_{k\bar l} ^{\quad
i\bar j}(g_t)\chi_{i\bar j}+g_t^{i\bar j}h^{r\bar q}\chi^{p\bar
s}\p_i \chi_{r\bar s}\p_{\bar j}\chi_{p\bar q}+g_t^{i\bar j}h^{p\bar
s}\chi^{r\bar q}\p_i \chi_{r\bar s}\p_{\bar j}\chi_{p\bar
q}\\&&-\chi^{k\bar l}R_{k\bar l}(g_t)-n\frac {|\td
\Na(\La_{\oo_t}\chi)
|^2}{ \La_{\oo_t}\chi }\Big)\\
&\geq&\frac 1{n\La_{\oo_t}\chi}\Big(h^{k\bar l}R_{k\bar l} ^{\quad
i\bar j}(g_t)\chi_{i\bar j}-\chi^{k\bar l}R_{k\bar l}(g_t)\Big),
\eeqs where we used the inequality by Lemma 3.2 in \cite{[W1]} \beq n|\td \Na(\La_{\oo_t}\chi)
|^2 \leq (\La_{\oo_t}\chi) g_t^{i\bar j}h^{r\bar q}\chi^{p\bar
s}\p_i \chi_{r\bar s}  \p_{\bar j}\chi_{p\bar q}. \eeq
 For any $A$, we have
 \beqs \td
\Delta \Big(\log (\La_{\oo_t}\chi)-A\varphi\Big)&\geq&\frac
1{n\La_{\oo_t}\chi}\Big(h^{k\bar l}R_{k\bar l} ^{\quad i\bar
j}(g_t)\chi_{i\bar j}-\chi^{k\bar l}R_{k\bar l}(g_t)\Big)-\frac
{1}{n}(nc_tA-A h^{k\bar l}  \chi_{0, k\bar l}). \eeqs Fix $t_0\in
[0, 1]$. We choose $A$ large enough such that \beq -\frac
1{A\La_{\oo_t}\chi}\Big(h^{k\bar l}R_{k\bar l} ^{\quad i\bar
j}(g_t)\chi_{i\bar j}-\chi^{k\bar l}R_{k\bar l}(g_t)\Big)\leq \ee,
\eeq where we used the fact that
$$\La_{\oo_t}\chi\geq C>0$$
for some constant $C$ independent of $t$. We assume that $\log
(\La_{\oo_t}\chi)-A\varphi$ achieves its maximum at the point $(x_t,
t).$ Then at the point $(x_t, t)$ we have \beqs
0&\geq& \td \Delta \Big(\log (\La_{\oo_t}\chi)-A\varphi\Big) \\
&\geq&\frac 1{n\La_{\oo_t}\chi}\Big(h^{k\bar l}R_{k\bar l} ^{\quad
i\bar j}(g_t)\chi_{i\bar j}-\chi^{k\bar l}R_{k\bar
l}(g_t)\Big)-\frac
{1}{n}(nc_tA-A h^{k\bar l} \chi_{0, k\bar l})\\
&\geq&\frac A{n}\Big(-\ee-  nc_t + h^{k\bar l} \chi_{0, k\bar
l}\Big). \eeqs Therefore, at the point $(x_t, t)$ we have \beqs
h^{k\bar l} \chi_{0, k\bar l} -nc_t \leq \ee. \eeqs We choose normal
coordinates for the metric $ \chi_{0}$ so that the metric $\chi_t$
is diagonal with entries $\la_1, \cdots, \la_n.$ We denote the
diagonal entries of $g_t$ by $\mu_1, \cdots, \mu_n. $ Thus, we have
$$ \sum_{i=1}^n\frac {\mu_i}{\la_i^2}-n c_t\leq \ee, $$
which implies the inequality \beqn \ee&\geq&\sum_{i=1, i\neq k}^n\,
\mu_i\Big(\frac 1{\la_i}-1\Big)^2-\sum_{i=1, i\neq k}^n\,
\mu_i +\frac {\mu_k}{\la_k^2}-2\frac {\mu_k}{\la_k}+nc_t\nonumber\\
&\geq & nc_t-\sum_{i=1, i\neq k}^n\, \mu_i-2\frac
{\mu_k}{\la_k}.\label{eq:mu} \eeqn
 On the other hand, by the assumption and the choice of the metric
 $\chi_0$ we have
\beq (nc \chi_0-(n-1)\oo)\wedge \chi_0^{n-2}\wedge
\bb_k>B\ee \chi_0^{n-1}\wedge \bb_k \label{eq:ct1} \eeq for
sufficiently small $\ee>0.$ Here   $\bb_k$ denotes the $(1, 1)$ form
$\i dz^k\wedge d\bar z^k$ and we choose the constant $B$ such that $
Bt_2>1. $
 Combining (\ref{eq:ct1})
with (\ref{eq:ct}), we have \beqn &&(nc_t
\chi_0-(n-1)\oo_t)\wedge \chi_0^{n-2}\wedge \bb_k \nonumber\\
&=&(1-t) \chi_0^{n-1}\wedge \bb_k+ t\Big(nc
\chi_0-(n-1)\oo\Big)\wedge \chi_0^{n-2}\wedge \bb_k\nonumber\\&\geq&
B\ee\, t\chi_0^{n-1}\wedge \bb_k. \label{eq:ct2}\eeqn
By the argument of \cite{[SW]},   the inequality (\ref{eq:ct2}) implies
that \beq \sum_{i=1, i\neq k}^n\, \mu_i<nc_t-Bt\ee. \label{eq:mu2}
\eeq Combining (\ref{eq:mu2}) with (\ref{eq:mu}), we have
$$\frac {\la_k}{\mu_k}<\frac 2{(Bt-1)\ee}\leq \frac 2{(Bt_2-1)\ee},
\quad t\geq t_2$$ for $k=1, \cdots, n.$ Hence, at the point $(x_t,
t)$ we have the estimate
$$\La_{\oo_t}\chi\leq \frac {2n}{(Bt_2-1)\ee}, \quad t\geq t_2.$$
Thus, for any $x\in X$ we have
$$\La_{\oo_t}\chi\leq \frac {2n}{(Bt_2-1)\ee} e^{A(\varphi-\inf_X\varphi)}.$$
Since $\oo_t\leq \chi_0+\oo\leq \la \oo$ and $\chi $ is a positive
$(1, 1)$ form, we have
$$\La_{\oo}\chi\leq \frac {2n\la }{(Bt_2-1)\ee} e^{A(\varphi-\inf_X\varphi)}. $$
The lemma is proved.

\end{proof}

\begin{lem}\label{lem:zero} Under the assumptions of Lemma \ref{lem:2order},
there exists a uniform constant $C$ such that
$$\osc_X\varphi_t\leq C,\quad \forall\;t\geq t_2. $$

\end{lem}
\begin{proof}The argument follows directly from \cite{[SW2]} and \cite{[SW]}
and we sketch the details here. We normalize $\varphi_t$ by
\beq
\int_X\;\varphi_t\,\oo_t^n=0.
\eeq
Therefore, we have $\sup_X\;\varphi_t\geq 0$ and $\inf_X\;\varphi_t\leq
0.$
Define
$$u=e^{-N\hat \varphi},\quad \hat \varphi=\varphi-\sup_X\varphi,$$
where $N=\frac A{1-\dd}.$ Here $A$ is the constant in Lemma \ref{lem:2order}
and $\dd$ is a small positive constant to be determined later. Using
the argument in \cite{[W1]}, there is a constant $C$ independent of $t$ such that for
all $p\geq 1$ and $t\geq t_2$ we have
$$\int_X\; |\Na u^{\frac p2}|_{\oo_t}^2\,\frac {\oo_t^n}{n!}\leq
C p\,\|u\|_{C^0}^{1-\dd}\,\int_X\, u^{p-(1-\dd)}\frac {\oo_t^n}{n!}.$$
 Note that the Sobolev constant of $\oo_t$ is uniformly
bounded for all $t\in [0, 1]$. Thus, following the argument of \cite{[W2]} we can prove that $u$ is bounded
and hence $\inf_X\hat \varphi_t\geq -C$ for some constant $C.$ In
other words, we have
$$0\geq \inf_X\varphi_t\geq \sup_X\;\varphi_t-C\geq -C,$$
which implies that $\osc_X\varphi_t\leq C. $ The lemma is proved.
\end{proof}

\end{document}